\documentclass[11pt]{amsart}
\usepackage{xcolor}

\include{package}
\definecolor{grn}{rgb}{0,0.6,0}
\definecolor{mrn}{rgb}{0.3,0,0}
\definecolor{blue}{rgb}{0,0,0.7}
\definecolor{Mygray}{rgb}{0.75,0.75,0.75}
\definecolor{auburn}{rgb}{0.43, 0.21, 0.1}
\definecolor{britishracinggreen}{rgb}{0.0, 0.26, 0.15}
\definecolor{taupe}{rgb}{0.28, 0.24, 0.2}
\usepackage{amsmath,amssymb}
\newtheorem{theorem}{Theorem}[section]

\newtheorem{propn}{Proposition}[section]
\newtheorem{cor}{Corollary}[section]

\newtheorem{quest}{Question}[section]
\newtheorem{rmk}{Remark}[section]

\newcommand{\eps}{\varepsilon}
\newcommand{\Z}{\mathbb{Z}}
\newcommand{\Q}{\mathbb{Q}}
\newcommand{\R}{\mathbb{R}}

\usepackage{latexsym,amssymb}
\usepackage{amssymb}
\usepackage{graphicx}
\usepackage{ textcomp }
\usepackage[left=2cm,right=2cm,top=2cm,bottom=2cm]{geometry}
\usepackage{stmaryrd}
\usepackage{tikz}
\usepackage{tikzpagenodes} 
\usepackage{lipsum}

\begin{document}
\baselineskip=14.5pt
\title[$2$-rank and structure of $2$-class groups]{On the structure and stability of ranks of $2$-class groups in cyclotomic $\mathbb{Z}_{2}$-extensions of certain real quadratic fields}

\author{Jaitra Chattopadhyay, H Laxmi and Anupam Saikia}
\address[Jaitra Chattopadhyay, H Laxmi and Anupam Saikia]{Department of Mathematics, Indian Institute of Technology Guwahati, Guwahati - 781039, Assam, India}
\email[Jaitra Chattopadhyay]{jaitra@iitg.ac.in; chat.jaitra@gmail.com}

\email[H Laxmi]{hlaxmi@iitg.ac.in}

\email[Anupam Saikia]{a.saikia@iitg.ac.in}

\begin{abstract}
For a real quadratic field $K= \mathbb{Q}(\sqrt{d})$ with discriminant $D_{K}$ having four distinct prime factors, we study the structure of the $2$-class group $A(K_{1})$ of the first layer $K_{1} = \mathbb{Q}(\sqrt{2},\sqrt{d})$ of the cyclotomic $\mathbb{Z}_{2}$-extension of $K$. With some suitably convenient assumptions on the rank and the order of $A(K_{1})$, we characterize $K$ for which the $2$-class group $A(K)$ is isomorphic to $\mathbb{Z}/2\mathbb{Z} \oplus \mathbb{Z}/2\mathbb{Z}$. We infer that the $2$-ranks of the class groups in each layer stabilizes by virtue of a result of Fukuda. This also provides an alternate way to establish that the Iwasawa $\mu$-invariant of $K$ vanishes. In some cases, we also provide sufficient conditions on the constituent prime factors of $D_{K}$ that imply $A(K) \simeq \mathbb{Z}/2\mathbb{Z} \oplus \mathbb{Z}/2\mathbb{Z}$, $A(K_{1}) \simeq \mathbb{Z}/2\mathbb{Z} \oplus \mathbb{Z}/4\mathbb{Z}$ and $A(K^{\prime}) \simeq \mathbb{Z}/2\mathbb{Z} \oplus \mathbb{Z}/2\mathbb{Z} \oplus \mathbb{Z}/2\mathbb{Z}$, where $K^{\prime} = \mathbb{Q}(\sqrt{2d})$. This extends some results obtained by Mizusawa.

\end{abstract}

\renewcommand{\thefootnote}{}



\footnote{\emph{We confirm that all the data are included in the article.}}

\renewcommand{\thefootnote}{\arabic{footnote}}
\setcounter{footnote}{0}
\maketitle

\section{Introduction}
The central question in classical Iwasawa theory is to study the growth of arithmetic objects, such as the ideal class groups of number fields, Selmer groups of elliptic curves defined over number fields etc., in an infinite tower of number fields over a fixed number field $K$. For a prime number $p$, an infinite Galois extension $K_{\infty}/K$ is said to be a $\Z_p$-extension if the Galois group $\Gamma := {\rm{Gal}}(K_{\infty}/K)$ is topologically isomorphic to the additive group $\mathbb{Z}_{p}$ of $p$-adic integers. It is well-known (cf. \cite[Proposition 13.1]{washington_book}) that for a given $\mathbb{Z}_{p}$-extension $K_{\infty}/K$ and for each integer $n \geq 0$, there exists a unique field $K_{n}$ with $K \subseteq K_{n} \subseteq K_{\infty}$ and $[K_{n} : K] = p^{n}$. The field $K_{n}$ is commonly termed as the $n^{\rm{th}}$ layer of the $\Z_{p}$-extension $K_{\infty}/K$. Let $A(K_n)$ be the $p$-Sylow subgroup of the ideal class group $Cl_{K_{n}}$ of $K_n$. Then these $A(K_{n})$'s form an inverse system under the norm map and the inverse limit $X = \lim\limits_{\substack{ \longleftarrow \\ n}}A(K_n)$ is called the {\it Iwasawa module} over the Iwasawa algebra $\Z_p \llbracket \Gamma \rrbracket$. 

\smallskip

In \cite{iwasawa}, Iwasawa proved a remarkable result about the order of $A(K_{n})$ and it is widely known as the {\it Iwasawa's class number formula}. It states that there exist constants $\mu(K_{\infty}), \lambda(K_{\infty})$ and $\nu(K_{\infty})$ such that for sufficiently large positive integers $n$, we have $\# A(K_n) = p^{\mu(K_{\infty}) p^n + \lambda(K_{\infty}) n + \nu(K_{\infty})}$. When $K$ is totally real, it has a unique $\Z_{p}$-extension which is given by the compositum of $K$ and the cyclotomic $\Z_{p}$-extension $\Q_{\infty}$ of $\Q$. In that case, we write $\mu$ (respectively, $\lambda$) in place of $\mu(K_{\infty})$ (respectively, $\lambda(K_{\infty})$). In \cite{greenberg}, Greenberg conjectured that the invariants $\mu$ and $\lambda$ must be equal to $0$ for totally real number fields and it was further proved by Ferrero and Washington in \cite{ferrero-washington} that the $\mu$-invariant always vanishes for the cyclotomic $\Z_{p}$-extension when the number field is abelian over the field $\Q$ of rational numbers. Various mathematicians have worked towards proving the vanishing of the $\lambda$-invariant for certain number fields where the fundamental discriminant has small number of prime factors (cf.  \cite{fukuda-komatsu}, \cite{kumakawa}, \cite{mizu_paper}, \cite{nishino}, \cite{ozaki-taya}, \cite{yamamoto}). Along a similar line, the first and the third author proved the vanishing of the Iwasawa $\lambda$-invariant for $p = 3$ for an infinite family of pairs of real quadratic fields (cf. \cite{asjc-rama}).

\smallskip

Apart from the Iwasawa invariants, the $p$-rank ${\rm{rk}}_{p}(Cl_{K})$ of $Cl_{K}$, which is essentially the maximal integer $r \geq 0$ such that $(\mathbb{Z}/p\mathbb{Z})^{r} \subseteq Cl_{K}$, sheds light on the structures of the class groups and their growths in the infinite tower. For a quadratic extension of number fields $K/k$, with the class number $h_{k}$ of $k$ being odd, Gras \cite{gras} found the $2$-rank of $Cl_{K}$ in certain cases by employing the Genus formulae. In \cite{bosma-stevenhagen}, Bosma and Stevenhagen derived an algorithm to calculate the $2$-class groups $A(K)$ (when $p=2$), using quadratic forms. In the cases of quadratic and multi-quadratic fields, the fundamental units have been extensively employed to retrieve information about the order and rank of $2$-class groups. Interested readers are encouraged to refer to \cite{azizi}, \cite{azizi2015}, \cite{azizi2019}, \cite{brown-parry}, \cite{mouhib-mova} and the references listed therein to find more information about the same.

\smallskip

An illuminating fact about a $\Z_p$-extension $K_{\infty}/K$ with intermediate fields $K_{n}$ is the existence of an integer $n_0$ such that every prime ramifying in $K_{\infty}/K_{n_0}$ is totally ramified (cf. \cite[Lemma 13.3]{washington_book}). In \cite[Theorem 1]{fukuda}, Fukuda found a criterion for the stabilization of ranks and orders of $p$-class groups if certain properties hold true from this stage $K_{n_{0}}$ onward. The particular case $p = 2$ draws special attention mainly because of the explicit forms of the intermediate layers of the cyclotomic $\mathbb{Z}_{2}$-extension. For each $n \in \mathbb{N}$, let $\zeta_{2^{n+2}}$ be a primitive ${2^{(n+2)}}^{\rm{th}}$ root of unity in the field of complex numbers. The $n^{\rm{th}}$ layer of the cyclotomic $\Z_2$-extension $\Q_{\infty}/\Q$ is the field $\mathbb{Q}_{n} := \Q( \zeta_{2^{n+2}} + \zeta^{-1}_{2^{n+2}})$. For a number field $K$, we define the cyclotomic $\Z_{2}$-extension of $K$ to be the compositum $K\Q_{\infty}$ and the $n^{\rm{th}}$ layer in the cyclotomic $\Z_2$-extension of $K$, denoted by $K_{n}$, is defined to be the compositum $K\Q_{n}$. In particular, the first layer is $K_1 = K(\sqrt{2})$. In the rest of this paper, since we deal with the cyclotomic $\Z_{2}$-extension of a real quadratic field $K$, in what follows, $A(K)$ always denotes the $2$-Sylow subgroup of $Cl_{K}$ and rank of any module stands for the $2$-rank, unless otherwise mentioned. 

\smallskip

Using Fukuda's result together with genus theory, Mizusawa \cite[Theorem 1]{mizu_paper} identified a class of real quadratic fields $K$ with $A(K) \simeq A(K_1) \simeq \Z/2\Z \oplus \Z/2\Z$. Motivated by Mizusawa's work, we ask the following question.

\begin{quest}\label{question-1}
Classify all the real quadratic fields $K = \mathbb{Q}(\sqrt{d})$ such that $A(K) \simeq \Z/2\Z \oplus \Z/2\Z$ and ${\rm {rank}}(A(K_1)) = 2$. In particular, characterize all the square-free integers $d > 0$ such that $A(K_1) \simeq \Z/2\Z \oplus \Z/4\Z$ and rank of ${\rm {rank}}A(K_n) = 2$ for all $n \geq 0$.
\end{quest}

In this paper, we try to answer Question \ref{question-1} by studying $K = \Q(\sqrt{d})$ where $d \geq 0$ is odd, square-free and has four distinct prime factors. We shall derive certain congruence conditions as well as Legendre symbol conditions on the prime factors of $d$, that provide us with an answer to  Question \ref{question-1}. More precisely, we prove the following theorems.

\begin{theorem}\label{rank stability}
Let $d \geq 1$ be a square-free integer, $K = \Q(\sqrt{d})$, $K' = \Q(\sqrt{2d})$ and $K_1 = \Q( \sqrt{2},\sqrt{d})$. Assume that the places above $2\mathcal{O}_K$ are totally ramified in $K_{1}$. Then ${\rm{rank}} \ A(K) = {\rm{rank}}  \ A(K_1) = 2$ and ${\rm{rank}} \ A(K') = 3$ if and only if $d$ is one of the following types.
\begin{enumerate}
\item $d = p_1p_2p_3$, where $p_{1}, p_{2}$ and $p_{3}$ are distinct primes with $p_{1} \equiv 1 \mbox{ or } 5 \pmod {8}$ and $p_{2} \equiv p_{3} \equiv 5 \pmod {8}$.
\item $d = p_1p_2q_1q_2$, where $p_{1}, p_{2}, q_{1}$ and $q_{2}$ are distinct primes with $p_1 \equiv p_2 \equiv 5 \pmod 8$, $q_1 \equiv \mbox{ either } 3 \mbox{ or } 7 \pmod 8$ and $q_2 \equiv 3 \pmod 8$.
\item $d = q_1q_2q_3q_4$, where $q_{1}, q_{2}, q_{3}$ and $q_{4}$ are distinct primes with $q_{1} \equiv \mbox 3 \mbox{ or } 7 \pmod {8}$ and $q_2 \equiv q_3 \equiv q_4 \equiv 3 \pmod 8$.
\end{enumerate}
\end{theorem}
 
We note that for fields in Theorem \ref{rank stability}, the ranks of the $2$-class groups of the consecutive layers in the cyclotomic $\Z_2$-extension of $K$ become equal. As a result we derive the following corollary.

\begin{cor}\label{cor to rank stability}
Let $K$ be a real quadratic field as mentioned in Theorem \ref{rank stability}. Then ${\rm{rank}} \ A(K_n) = 2$ for all integer $n \geq 0$ and ${\rm{rank}} \ A(K') = 3$.
\end{cor}

Our next two theorems deal with those real quadratic fields whose discriminants consist of at least one prime divisor that is congruent to $7$ modulo $8$. We state the theorems as follows.

\begin{theorem}\label{A(K1) for K = p1p2q1q2)}
Let $K = \Q(\sqrt{d})$ be a real quadratic field with $d =p_1p_2q_1q_2$, where $p_{1}, p_{2}, q_{1}$ and $q_{2}$ are distinct primes, $p_1 \equiv p_2\equiv 5 \pmod 8, q_1 \equiv 7 \pmod 8, q_2 \equiv 3 \pmod 8$. Let $K' = \Q(\sqrt{2d})$ and $K_{1} = \Q(\sqrt{2},\sqrt{d})$. Then $A(K) \simeq \Z/2\Z \oplus \Z/2\Z$, $A(K') \simeq \Z/2\Z \oplus \Z/2\Z \oplus \Z/2\Z$ and $A(K_1) \simeq \Z/2\Z \oplus \Z/4\Z$ if and only if one of the following choices of Legendre symbols holds.
\begin{enumerate}
\item $\left( \dfrac{p_1}{p_2} \right) = -1$, $\left( \dfrac{p_1}{q_1} \right) = -1$, $\left( \dfrac{p_1}{q_2} \right) = 1$, $\left( \dfrac{q_1q_2}{p_2} \right) = 1$,
\item $\left( \dfrac{p_1}{p_2} \right) = -1$, $\left( \dfrac{q_1q_2}{p_1} \right) = 1$, $\left( \dfrac{p_2}{q_1} \right) = -1$, $\left( \dfrac{p_2}{q_2} \right) = 1$,
\item $\left( \dfrac{p_1}{p_2} \right) = 1$, $\left( \dfrac{p_1p_2}{q_1} \right) = -1$, $\left( \dfrac{p_1p_2}{q_2} \right) = -1$,  $\left( \dfrac{p_1}{q_1} \right) = \left( \dfrac{p_2}{q_2} \right) $.
\end{enumerate}
\end{theorem}

\begin{theorem}\label{A(K1) for K = q1q2q3q4}
Let $K = \Q(\sqrt{d})$ be a real quadratic number field with $d = q_1q_2q_3q_4$, where $q_{1}, q_{2}, q_{3}$ and $q_{4}$ are distinct primes with $q_1 \equiv 7 \pmod 8$, $q_2 \equiv q_3 \equiv q_4 \equiv 3 \pmod 8$. Let  $K' = \Q(\sqrt{2d})$ and $K_{1} = \Q(\sqrt{2},\sqrt{d})$. Then $A(K) \simeq \Z/2\Z \oplus \Z/2\Z$, $A(K') \simeq \Z/2\Z \oplus \Z/2\Z \oplus \Z/2\Z$ and $A(K_1) \simeq \Z/2\Z \oplus \Z/4\Z$ if one of the following choices of Legendre symbols holds: 
\begin{enumerate}
\item $\left( \dfrac{q_1}{q_3}\right)=\left( \dfrac{q_2}{q_3}\right)=\left( \dfrac{q_4}{q_2}\right)=\left( \dfrac{q_4}{q_1}\right)=1, \left( \dfrac{q_4}{q_3}\right) =-1$,
\item $\left( \dfrac{q_1}{q_3}\right)=\left( \dfrac{q_2}{q_3}\right)=\left( \dfrac{q_4}{q_2}\right)=\left( \dfrac{q_4}{q_1}\right)=-1, \left( \dfrac{q_4}{q_3}\right) =1$,
\item $\left( \dfrac{q_1q_2}{q_3}\right)=-1, \left( \dfrac{q_1q_2}{q_4}\right)=-1, \left( \dfrac{q_2}{q_3}\right) = \left( \dfrac{q_1}{q_4}\right)= \left( \dfrac{q_3}{q_4}\right)$,
\item $\left( \dfrac{q_1q_2}{q_3}\right)=1, \left( \dfrac{q_1q_2}{q_4}\right)=-1, \left( \dfrac{q_2}{q_3}\right) = \left( \dfrac{q_2}{q_4}\right), \left( \dfrac{q_1}{q_2}\right)=\left( \dfrac{q_4}{q_3}\right)$,
\item $\left( \dfrac{q_1q_2}{q_3}\right)=-1, \left( \dfrac{q_1q_2}{q_4}\right)=1, \left( \dfrac{q_2}{q_3}\right) = \left( \dfrac{q_1}{q_4}\right), \left( \dfrac{q_1}{q_2}\right) = \left( \dfrac{q_3}{q_4}\right)$,
\item $\left( \dfrac{q_1}{q_3}\right)=\left( \dfrac{q_2}{q_3}\right)= 1, \left( \dfrac{q_1}{q_4}\right)=1, \left( \dfrac{q_2}{q_4}\right)=-1, \left( \dfrac{q_1}{q_2}\right)=-1$,
\item $\left( \dfrac{q_1}{q_3}\right)=\left( \dfrac{q_2}{q_3}\right)= -1, \left( \dfrac{q_1}{q_4}\right)=-1, \left( \dfrac{q_2}{q_4}\right)=1, \left( \dfrac{q_1}{q_2}\right)=1,  \left( \dfrac{q_3}{q_4}\right)=1$,
\item $\left( \dfrac{q_1}{q_3}\right)= 1, \left( \dfrac{q_2}{q_3}\right)= -1, \left( \dfrac{q_1}{q_4}\right)=1, \left( \dfrac{q_2}{q_4}\right)=1, \left( \dfrac{q_1}{q_2}\right)=-1$,
\item $\left( \dfrac{q_1}{q_3}\right)= -1, \left( \dfrac{q_2}{q_3}\right)= 1, \left( \dfrac{q_1}{q_4}\right)=-1, \left( \dfrac{q_2}{q_4}\right)=-1, \left( \dfrac{q_1}{q_2}\right)=1, \left( \dfrac{q_3}{q_4}\right)=-1$.
\end{enumerate}
\end{theorem}

\begin{rmk}
The arguments that will be used in the proofs of Theorem \ref{A(K1) for K = p1p2q1q2)} and Theorem \ref{A(K1) for K = q1q2q3q4} are mostly same and therefore, we shall furnish the proof only for Theorem \ref{A(K1) for K = p1p2q1q2)}.
\end{rmk}

\section{Preliminaries}
Since we wish to find the ranks and the orders of the $2$-class groups of certain real quadratic fields, the {\it Genus formula}, which relates the order of the $2$-class group in a quadratic extension of number fields to the number of ramified primes in that extension, will play a crucial role in the proofs of the theorems. We record it as follows.
 
\begin{theorem}(Genus Formulae) \cite[Theorem 2.5]{mizu_thesis}\label{genusfor}
Let $K/k$ be a quadratic extension of number fields with Galois group $G = {\rm{Gal}} \left(K/k\right)$. Let $A(K)^{G}$ be the subgroup of $A(K)$ generated by the ideal classes in $A(K)$ that are fixed by the action of $G$ on $A(K)$. Let $N_{K/k}$ stand for the norm map from $K$ to $k$. Let $E(K)$ and $E(k)$ be the unit groups of $K$ and $k$, respectively. If $t$ is the number of places of $k$ ramified in $K$, then the following equality holds.
\begin{align*}
\#A(K)^{G} &= \#A(k) \times \dfrac{2^{t-1}}{\left[ E(k):E(k)\cap N_{K/k}K^{\times} \right]}.
\end{align*}
\end{theorem}

\begin{rmk}\cite[Remark 2.6]{mizu_thesis}\label{rmk to genus}
If $K/k$ is a quadratic extension of number fields and the image of the lifting map $j: A(k) \rightarrow A(K)$ is trivial, then the non-trivial element of ${\rm{Gal}}(K/k)$ acts as $-1$ on $A(K)$. In that case, $A(K)^{G}$ is generated by the elements in $A(K)$ of order $2$. Consequently, $\# A(K)^{G} = \# \left( A(K)/2A(K)\right) = 2^{\ {\rm{rank}} \ A(K)}$.
\end{rmk}

\begin{rmk}\label{rmk2 to genus}
Let $K = \Q(\sqrt{d})$ be a real quadratic field and assume that ${\rm{rank}} \ A(K)=2$. Since $A(\Q)$ is trivial, the lifting map $j: A(\Q) \rightarrow A(K)$ is trivial. By Remark \ref{rmk to genus}, we obtain $\#A(K)^{G} = 2^2 = 4$ and by Theorem \ref{genusfor}, we have
\begin{equation}\label{remarkwala}
\#A(K)^{G} = \dfrac{\#A(\Q)\cdot 2^{t-1}}{\left[E(\Q):E(\Q)\cap N_{K/\Q}(K^{\times})\right]}.
\end{equation}
Since $E(\Q) = \{-1, 1\}$, the denominator in \eqref{remarkwala} is either $1$ or $2$. As a result, we have $4 = \dfrac{2^{t-1}}{n}$, where $n =1$ or $2$. Therefore, $t=3$ or $4$.
\end{rmk}
In the case of a quadratic number field $K$, its genus field $K_{G}$, which is defined as the maximal abelian extension of $\Q$ contained inside the Hilbert class field $H_{K}$ of $K$, plays a pivotal role in determining the $2$-rank of $Cl_{K}$. The genus field can be obtained from the narrow genus field $K_G^{+}$ from the relation that when $K$ is imaginary, we have $K_G = K_G^{+}$ and when $K$ is real, we have $K_G = K_G^{+} \cap \mathbb{R}$. The narrow genus field of $K = \Q(\sqrt{d})$ with discriminant $D_{K}$ can be computed by the following algorithm. We express the prime factorization of $D_{K}$ by $D_{K} = \pm 2^{e}p_1^{\ast}\cdots p_t^{\ast}$, where 
\begin{equation*}
  p_i^{\ast} = 
      \begin{cases}
        p_i & \text{if} \  p_i \equiv 1 \pmod 4\\
        -p_i & \text{if}\ p_i \equiv 3 \pmod 4.
      \end{cases}     
\end{equation*}
With these $p_i^{\ast}$'s, we have $K_G^{+} = \mathbb{Q}( \sqrt{d}, \sqrt{p_1^{\ast}}, \cdots, \sqrt{p_t^{\ast}})$. 
 
From this expression, we observe that ${\rm{rank}} \ {\rm{Gal}}(K_G /K) = {\rm{rank}} \ A(K)$ (cf. \cite[Section 2.2]{mizu_thesis}). It is an elementary fact that an abelian $2$-group $G$ is $2$-elementary if and only if $ \#\left(2G/4G\right) = 1$. Let $A^{+}(K)$ be the $2$-Sylow subgroup of the narrow class group $Cl_{K}^{+}$ of $K$. A theorem of R\'{e}dei and Reichardt (cf. \cite{redei-reichardt}, \cite[Theorem 2.4]{mizu_thesis}) connects $\# (2A^{+}(K)/4A^{+}(K))$, to the number of ways of expressing $D_{K}$ as the product of two factors satisfying certain criteria. Let $S_1(K)$ and $S_2(K)$ be the sets of tuples $(D_1, D_2)$ defined as follows:
\begin{align*}
S_1(K) &:= \{ (D_1, D_2): |D_1| < |D_2|,\ D_{K}= D_1D_2,\ D_i \equiv 0 \mbox{ or } 1\pmod 4 \}, \\
T_1(K) &:= \{(D_1, D_2) \in S_1(K):\chi_{D_1}(p) = 1 \text{ for all prime $p$ dividing $D_2$}\},\\
T_2(K) &:= \{(D_1, D_2) \in S_1(K): \chi_{D_2}(p) = 1 \text{ for all prime $p$ dividing $D_1$}\},\\
S_2(K) &:= \{(1, D_{K})\} \ \cup \left( T_1(K) \cap T_2(K) \right),
\end{align*}
where $\chi_{D_{i}}(p) = \left(\dfrac{D_{i}}{p}\right)$ is the Kronecker symbol, for $i = 1 \mbox{ and } 2$. 

\begin{theorem}(\cite{redei-reichardt}, \cite[Theorem 2.4]{mizu_thesis})\label{RR}
With the quantitites defined above, we have $$\# S_1(K) = \#\left(A^{+}(K) / 2A^{+}(K)\right) ~ \mbox{ and }  ~ \#S_2(K) = \# \left(2A^{+}(K) / 4A^{+}(K)\right).$$
\end{theorem}
\begin{rmk}
We observe from Theorem \ref{RR} that $A^{+}(K)$ is $2$-elementary if and only if $\#S_2(K) = 1$. Also in that case, we have $\#S_1(K) = \#A^{+}(K)$.  
\end{rmk}  
Since $Cl_{K}$ is a subgroup of the narrow class group $Cl^{+}_{K}$, the $2$-Sylow subgroup $A(K)$ of $Cl_{K}$ is a subgroup of the $2$-Sylow subgroup $A^{+}(K)$ of $Cl^{+}_{K}$. Therefore, if $A^{+}(K)$ is a $2$-elementary group, then so is $A(K)$. The next proposition provides a sufficient condition for the converse to hold.
\begin{propn}\label{A(K) and A+(K) are 2 elmtry}
Let $K = \Q(\sqrt{d})$ be a quadratic field, where that $d \geq 1$ is a square-free integer having a prime divisor which is congruent to $3 \pmod 4$. If $A(K)$ is $2$-elementary, then so is $A^{+}(K)$.
\end{propn}

\begin{proof}
Let $r \equiv 3 \pmod {4}$ be a prime divisor of $d$. First, we claim that the fundamental unit $\varepsilon = \dfrac{a + b\sqrt{d}}{2}$ of $K$, where $a$ and $b$ are rational integers of same parity, is of absolute norm $N_{K/\Q}(\varepsilon) = 1$. For if $N_{K/\Q}(\varepsilon) = -1$, then reading the equation $N_{K/\Q}(\varepsilon) = \dfrac{a^{2} - db^{2}}{4} = -1$ modulo $r$, we obtain $a^{2} \equiv -4 \pmod {r}$. This implies that $-1$ is a quadratic residue $\pmod {r}$, which is impossible since $r \equiv 3 \pmod {4}$. Therefore, $N_{K/\Q}(\varepsilon) = 1$.

\smallskip

Assume that $d$ is odd with the prime factorization $d = p_1\cdots p_m \cdot q_1\cdots q_n$, where $p_i \equiv 1 \pmod 4$ and $q_j \equiv 3 \pmod 4$. Since the rest of the argument is similar for $n$ being even or odd, we furnish the proof assuming that $n$ is even. Then $-q_n = \dfrac{d}{\prod\limits_{i=1}^{m}p_i\prod\limits_{j=1}^{n-1}(-q_j)}$. Thus we have $$K_G^{+} = \Q( \sqrt{d}, \sqrt{p_1},\ldots,\sqrt{p_m},\sqrt{-q_1},\ldots, \sqrt{-q_{n-1}})$$ and $$K_{G} = K_{G}^{+} \cap \R = \Q(\sqrt{d},\sqrt{p_{1}},\ldots,\sqrt{p_{m}},\sqrt{q_{1}q_{2}},\ldots,\sqrt{q_{1}q_{n - 1}}).$$ Consequently, $\left[K_G:K \right] = 2^{m+n-2}$ and $[K_{G}^{+} : K_{G}] = 2$. Since ${\rm{Gal}}(K_G/\Q)$ is $2$-elementary, so is ${\rm{Gal}}(K_G/K)$. Hence ${\rm{rank}} \ A(K) = {\rm{rank}} \ {\rm{Gal}}(K_G/K) = m+n-2$. According to our hypothesis, $A(K)$ is $2$-elementary and therefore, $A(K) \simeq \displaystyle\bigoplus_{m+n-2}\Z/ 2\Z$, and hence $\left[ L(K):K\right] = 2^{m+n-2}$, where $L(K)$ is the $2$-Hilbert class field of $K$. We see that $K \subseteq K_G$ and $L(K) \subseteq H_{K}$, where both $K_G$ and $L(K)$ have the same degree over $K$, which equals $2^{m+n-2}$. Since $L(K)$ is the subfield of $H_{K}$ with maximal $2$-power degree over $K$, we conclude that $L(K)=K_G$.

\smallskip

Since $d \geq 1$ is square-free and $N_{K/\Q}(\eps) =1$, by a standard result (cf. \cite[Exercise 3.2]{nancy_childress}), we have $\# Cl^{+}_{K} = 2 \times \#Cl_{K}$ and thus $\# A^{+}(K) = 2 \times \# A(K) = 2^{m+n-1} = [L^{+}(K) : K]$. This indicates that $L^{+}(K)$, which is contained inside the narrow Hilbert class field of $K$ and is the fixed field corresponding to $A^{+}(K)$, and $K^{+}_G$ have the same degree over $K$. Once again, we obtain $L^{+}(K) = K^{+}_G$. Since ${\rm{Gal}} (K^{+}_G/K)$ is $2$-elementary, therefore $A^{+}(K) = {\rm{Gal}}(L^{+}(K)/K)$ must be $2$-elementary. The proof follows a similar line of argument for ${\rm{rank}} \ A(K) = m+n-1$ and also when $2$ divides $d$.
\end{proof}
Azizi and Mouhib \cite{azizi}, and subsequently Mizusawa \cite{mizu_thesis} established a criterion on the $2$-rank of the ideal class group of totally real bi-quadratic fields. This enables us to calculate the $2$-rank of $A(K_1)$.

\begin{theorem}(\cite{azizi}, \cite[Theorem 2.7]{mizu_thesis})\label{rank of biq}
Let $K = \mathbb{Q}(\sqrt{d})$ be a real quadratic field where $d \geq 1$ is an odd square-free integer. Let $t_1$ be the number of places of $\mathbb{Q}(\sqrt{2})$ ramified in $K_1 = \Q(\sqrt{2},\sqrt{d})$, and let $r_1$ be the $2$-rank of $A(k_1)$. Then the following hold.
\begin{enumerate}
\item If $d$ has a prime factor congruent to $3 \pmod 4$, then either $r_1 = t_1 -2$ or $r_1 = t_1 - 3$. In particular, $r_1 = t_1 - 2$ holds if and only if $d$ has no prime factor which is congruent to $7 \pmod 8$.
\item If $d$ has no prime factor congruent to $3 \pmod 4$, then either $r_1 = t_1 -1$ or $r_1 = t_1 - 2$. In particular, $r_1 = t_1 -1$ holds if and only if $d$ has no prime factor $p$ such that $p \equiv 1 \pmod 8$ and $2^{\frac{p-1}{4}} \not\equiv (-1)^{\frac{p-1}{8}} \pmod p$.
\end{enumerate}
\end{theorem}

Also for bi-quadratic fields, Kuroda \cite{kuroda}, and then Kubota \cite{kubota} found an explicit formula for the order of the $2$-class group in terms of the orders of $2$-class groups of the quadratic subfields. This formula is known as {\it ``Kuroda's class number formula"}. We recall it in the following theorem. 
\begin{theorem}(\cite{kubota}, \cite{kuroda}, cf. \cite[Proposition 1]{mizu_paper})\label{kubota}
Let $K/\mathbb{Q}$ be a totally real bi-quadratic bi-cyclic extension, with unit group $E(K)$. Let $F_1, F_2$ and $F_3$ be the quadratic subfields of $K$. Let $\eps_i$ be the fundamental unit of $F_i$, for $i=1,2$ and $3$. Let $Q(K) := [E(K) : \langle -1, \eps_{1}, \eps_{2}, \eps_{3} \rangle]$ be the Hasse unit index of $K$. Then we have
\begin{equation}
\#A(K) = \dfrac{1}{4}\cdot Q(K)\cdot \#A(F_1)\cdot \#A(F_2)\cdot \#A(F_3).
\end{equation}
Then a system $\mathcal{F}$ of fundamental units of $K$ is one of the following possibilities.
\begin{enumerate}
\item If $N_{F_{1}/\mathbb{Q}}(\eps_1) = 1$, then $\mathcal{F} = \{\eps_1, \eps_2, \eps_3\}$ or $\mathcal{F} = \{\sqrt{\eps_1}, \eps_2, \eps_3\}$.
  
\item If $N_{F_{1}/\mathbb{Q}}(\eps_1) = N_{F_{2}/\mathbb{Q}}(\eps_2) = 1$, then $\mathcal{F} = \{\sqrt{\eps_1}, \sqrt{\eps_2}, \eps_3\}$ or $\mathcal{F} = \{\sqrt{\eps_1 \eps_2}, \eps_2, \eps_3\}$.
   
\item If $N_{F_{1}/\mathbb{Q}}(\eps_1) = N_{F_{2}/\mathbb{Q}}(\eps_2) = N_{F_{3}/\mathbb{Q}}(\eps_3) = 1$, then $\mathcal{F} = \{\sqrt{\eps_1\eps_2}, \eps_2, \sqrt{\eps_3}\}$ or\\ $\mathcal{F} = \{\sqrt{\eps_1 \eps_2},\sqrt{\eps_2 \eps_3}, \sqrt{\eps_3 \eps_3}\}$.

\item If $N_{F_{1}/\mathbb{Q}}(\eps_1) = N_{F_{2}/\mathbb{Q}}(\eps_2) = N_{F_{3}/\mathbb{Q}}(\eps_3) = \pm 1$, then $\mathcal{F} = \{\sqrt{\eps_1 \eps_2 \eps_3}, \eps_2, \eps_3\}$.  \end{enumerate}
\end{theorem}

The following theorem by Fukuda \cite{fukuda} provides certain conditions for the stability of rank and order of $A(K_n)$ for all integer $n \geq 1$.

\begin{theorem}\cite[Theorem 1]{fukuda}\label{fukuda's result}
Let $p$ be a prime number. Let $k$ be a number field and let $k_{\infty}/k$ be a $\Z_p$-extension of $k$. Let $n_0 \geq 0$ be an integer such that any prime of $k_{\infty}$ which is ramified in $k_{\infty}/k$ is totally ramified in $k_{\infty}/k_{n_0}$. Denote the $n^{\rm{th}}$ layer of $k_{\infty}/k$ by $k_{n}$ and the $p$-Sylow subgroup of $Cl_{k_{n}}$ by $A(k_{n})$. Then the following hold.
\begin{enumerate}
\item If there exists an integer $n \geq n_0$ such that $\#A(k_{n+1}) = \#A(k_n)$, then $\#A(k_m) = \#A(k_n)$ for all $m \geq n$. In particular, both the Iwasawa $\mu$ and $\lambda$-invariants vanish. 
  
  \smallskip
  
\item If there exists an integer $n \geq n_0$ such that ${\rm{rk}}_{p}(A(K_{n+1})) = {\rm{rk}}_{p}(A(K_{n}))$, then ${\rm{rk}}_{p}(A(K_{m})) = {\rm{rk}}_{p}(A(K_{n}))$ for all $m \geq n$. In particular, the Iwasawa $\mu$-invariant vanishes.
\end{enumerate}
\end{theorem}


The next proposition asserts the existence of infinitely many rational primes in prescribed arithmetic progressions satisfying any given set of conditions on the Legendre symbols. More precisely, the proposition is as follows.
\begin{propn}\label{infinite primes}
Let $t \geq 1$ be an integer. Assume that for each $i \in \{1,\ldots,t\}$, we are given integers $a_i \in \{1,3,5,7 \}$, and for each $1 \leq j < k \leq t$, the integers $\eps_{kj} \in \{ \pm 1 \}$ are specified. Then there exist infinitely many $t$-tuples $\{ p_1,\ldots,p_t \}$ of prime numbers such that $p_i \equiv a_i \pmod {8}$ and the Legendre symbol $\left( \dfrac{p_k}{p_j} \right)$ equals $\eps_{kj}$.
\end{propn}
\begin{proof}
We prove this by induction on $t$. For $t = 2$, we assume that $p_1 \equiv a_1 \pmod 8$ is given. We wish to find $p_2 \equiv a_2 \pmod 8$ such that $\left( \dfrac{p_2}{p_1} \right) = \eps_{21}$. Let $1 \leq v \leq p_1 -1$ be an integer such that $\left( \dfrac{v}{p_1} \right) = \eps_{21}$. Consider the system of congruences
\begin{align*}
X  &\equiv a_2 \pmod 8 \\
X  &\equiv v \pmod {p_1}.
\end{align*}
Then by the Chinese Remainder Theorem, there exists a unique solution $x_0 \pmod {8p_1}$ to this system. Therefore, $\gcd ( x_0, 8p_1 ) =1$ and consequently, by Dirichlet's theorem for primes in an arithmetic progression, there exist infinitely many primes $\ell \equiv x_0 \pmod {8p_1}$. Then $\ell \equiv x_0 \equiv a_2 \pmod 8$ and $\left( \dfrac{\ell}{p_1} \right) = \left( \dfrac{x_0}{p_1} \right) = \left( \dfrac{v}{p_1} \right) = \eps_{21}$. Thus the statement holds true for $t =2$.

\smallskip

Now, we assume that the proposition holds true for $t-1$. That is, for given integers $a_i \in \{1,3,5,7 \}$, $1 \leq i \leq t-1$ and given integers $\eps_{kj} \in \{ \pm 1 \}$ with $j < k$, there exist infinitely many $(t-1)$-tuples $\{ p_1 \cdots p_{t-1} \}$ of rational primes satisfying the hypotheses of the proposition. Now, for a given integer $a_t \in \{1,3,5,7 \}$ and given integers $\eps_{tk} \in \{ \pm 1 \}$ for $k \in \{1,\cdots ,t-1\}$, let us consider the following system of congruences.
\begin{align*}
X  &\equiv a_t \pmod {8} \\
X  &\equiv v_1 \pmod {p_1}\\
\vdots\\
X  &\equiv v_{t-1} \pmod {p_{t-1}},
\end{align*}
where $p_{1},\ldots,p_{t - 1}$ is a $(t - 1)$-tuple of prime numbers satisfying the induction hypothesis. Let $1 \leq v_j \leq p_j-1$ be such that $\left( \dfrac{v_j}{p_j} \right) = \eps_{tj}$. Again by the Chinese Remainder Theorem, this system has a unique solution $y_0 \pmod {8p_1\cdots p_{t-1}}$. Since $\gcd ( y_0, 8p_1\cdots p_{t-1}) =1$, by Dirichlet's theorem for primes in an arithmetic progression, we have infinitely many primes $\ell \equiv y_0 \pmod {8p_1\cdots p_{t-1}}$. Then $\ell \equiv y_0 \equiv a_t \pmod 8$ and  $\left( \dfrac{\ell}{p_j} \right) = \left( \dfrac{y_0}{p_j} \right) = \left( \dfrac{v_j}{p_j} \right) = \eps_{tj}$. This completes the proof of the proposition.
\end{proof}
\begin{rmk}
    In view of Proposition \ref{infinite primes}, we see that there are infinitely many tuples of prime numbers satisfying the required Legendre symbol conditions of Theorem \ref{rank stability}, Theorem \ref{A(K1) for K = p1p2q1q2)} and Theorem \ref{A(K1) for K = q1q2q3q4}.
\end{rmk}
\section{Proof of Theorem \ref{rank stability}}

Let $K = \Q(\sqrt{d})$ be a real quadratic field. First, we assume that ${\rm{rank}} \ A(K) = {\rm{rank}} \ A(K_{1}) = 2$ and ${\rm{rank}} \ A(K^{\prime}) = 3$. By Remark \ref{rmk2 to genus}, the discriminant $D_{K}$ can have either $3$ or $4$ prime factors. Among the odd prime divisors of $D_{K}$, let $p_i$ denote those rational primes that are congruent to $1 \pmod 4$ and let $q_j$ denote those that are congruent to $3 \pmod 4$. We now enlist the suitable choices of $d$ in Table 1 and Table 2. 
\begin{table}[hbt!]
\begin{center}
\caption{Possibilities of $K$ when $t=3$}
\smallskip
\label{tab:table1}
\begin{tabular}{|c| c| c| c|} 
\hline 
Ramified primes & $K = \Q(\sqrt{d})$ & $d \pmod 4$ & $D(K)$  \\
\hline
$2, p_1, p_2$ & $\Q(\sqrt{2p_1p_2})$ & $2$ & $8p_1p_2$  \\
$2, p_1, q_1$ & $\Q(\sqrt{p_1q_1})$  & $3$ & $4p_1q_1$\\
$2, p_1, q_1$ & $\Q(\sqrt{2p_1q_1})$ & $2$ & $8p_1q_1$\\
$2, q_1, q_2$ & $\Q(\sqrt{2q_1q_2})$ & $2$ & $8q_1q_2$\\
$p_1,p_2,p_3$ & $\Q(\sqrt{p_1p_2p_3})$ & $1$ & $p_1p_2p_3$\\
$p_1,q_1,q_2$ & $\Q(\sqrt{p_1q_1q_2})$ & $1$ & $p_1q_1q_2$\\
\hline
\end{tabular}
\end{center}
\end{table}

\begin{table}[hbt!]
\begin{center}
\caption{Possibilities of $K$ when $t=4$}
\smallskip
\label{tab:table2}
\begin{tabular}{|c| c| c| c|} 
\hline 
Ramified primes & $K = \Q(\sqrt{d})$ & $d \pmod 4$ & $D(K)$  \\
\hline
$2, p_1, p_2, p_3$ & $\Q(\sqrt{2p_1p_2p_3})$ & $2$ & $8p_1p_2p_3$  \\
$2, p_1, p_2, q_1$ & $\Q(\sqrt{p_1p_2q_1})$  & $3$ & $4p_1p_2q_1$\\
$2, p_1, p_2, q_1$ & $\Q(\sqrt{2p_1p_2q_1})$ & $2$ & $8p_1p_2q_1$\\
$2, p_1, q_1, q_2$ & $\Q(\sqrt{2p_1q_1q_2})$ & $2$ & $8p_1q_1q_2$\\
$2, q_1, q_2, q_3$ & $\Q(\sqrt{q_1q_2q_3})$  & $3$ & $4q_1q_2q_3$\\
$2, q_1, q_2, q_3$ & $\Q(\sqrt{2q_1q_2q_3})$ & $2$ & $8q_1q_2q_3$\\
$p_1,p_2, p_3, p_4$ & $\Q(\sqrt{p_1 p_2 p_3 p_4})$ & $1$ & $p_1 p_2 p_3 p_4$\\
$p_1,p_2, q_1,q_2$ & $\Q(\sqrt{p_1 p_2 q_1 q_2})$ & $1$ & $p_1 p_2 q_1 q_2$\\
$q_1,q_2, q_3,q_4$ & $\Q(\sqrt{q_1 q_2 q_3 q_4})$ & $1$ & $q_1 q_2 q_3 q_4$\\
\hline
\end{tabular}
\end{center}
\end{table}

We shall restrict ourselves to only those $K$ where the prime above $2$ gets totally ramified in $K_1/K$. Thus we discard the fields $K = \Q(\sqrt{2p_1p_2}), \Q(\sqrt{2q_1q_2}), \Q(\sqrt{2p_1p_2p_3}),$ and $\Q(\sqrt{2p_1q_1q_2})$. We observe that the genus field $K_G$ for $K = \Q(\sqrt{2p_1q_1})$ is $\Q(\sqrt{2p_1q_1}, \sqrt{p_1})$. Therefore, ${\rm{Gal}}(K_G/K) \cong \Z/2\Z$. Thus, ${\rm{rank}} \ (A(K))=1$, which is not of our consideration. The same happens with $\Q(\sqrt{p_1q_1})$ and $\Q(\sqrt{p_1q_1q_2})$ as well. Similarly, the genus field of $K = \Q(\sqrt{p_1p_2p_3p_4})$ is $\Q(\sqrt{p_1},\sqrt{p_2},\sqrt{p_3},\sqrt{p_4})$ and ${\rm{rank}} \ A(K) = 3$. The fields $\Q(\sqrt{q_1q_2q_3}), \Q(\sqrt{2q_1q_2q_3}), \Q(\sqrt{p_1p_2q_1})$ and $\Q(\sqrt{2p_1p_2q_1})$ have already been covered by Mizusawa in \cite{mizu_thesis} and \cite{mizu_paper}. Therefore, we are left only with the fields $\Q(\sqrt{p_1p_2q_1q_2})$, $\Q(\sqrt{q_1q_2q_3q_4})$ and $\Q(\sqrt{p_1p_2p_3})$. Now, we assume that ${\rm{rank}} \ A(K_1)=2$ to find more information about the primes $p_i$'s and $q_j$'s.

\smallskip

Let $K = \Q(\sqrt{p_1p_2q_1q_2})$. Since $d \equiv 1 \pmod 4$, the place above $2$ is unramified in $K$ but it is ramified in $\Q(\sqrt{2})$. Since ramification index is multiplicative in a tower of number fields, we conclude that the place above the rational prime $2$ in $\Q(\sqrt{2})$ must be unramified in $K_1$. We now appeal to the Case $1$ of Theorem \ref{rank of biq}, and taking $r_1= {\rm{rank}} \ A(K_1) =2$, we consider the following two cases.


\smallskip

\textbf{Case 1.} $r_1 = t_1 -2$. In this case, $d$ must not have a prime factor congruent to $7 \pmod 8$. This gives us that $q_1 \equiv q_2 \equiv 3 \pmod 8$. Also, $t_1 = 4$ implies that exactly $4$ places of $\Q(\sqrt{2})$ must be ramified in $K_1$. Since $q_1 \equiv q_2 \equiv 3 \pmod 8$, the primes above $q_1$ and $q_2$ must be inert in $\Q(\sqrt{2})$. Thus there is exactly one prime in $\Q(\sqrt{2})$ lying above $q_j \ (j=1,2)$ which must be ramified in $K_1$. Now if for some $i$, $p_i \equiv 1 \pmod 8$, then $p_i$ must totally split in $\Q(\sqrt{2})$. This contributes two places above $p_i$, making the total number of places in $\Q(\sqrt{2})$ ramified in $K_1$ at least 5, which is not possible. Hence $p_1 \equiv p_2 \equiv 5 \pmod 8$. Combining all these, we get $p_1 \equiv p_2 \equiv 5 \pmod 8, q_1 \equiv q_2 \equiv 3 \pmod 8$.

\smallskip

\textbf{Case 2.} $r_1 = t_1 -3$. In this case, $t_1 = 5$ and hence at least one of the $q_j$'s must be congruent to $7 \pmod 8$. If $q_j \equiv 7 \pmod 8$, then it must be totally split in $\Q(\sqrt{2})$, yielding two places above $q_j$. Each $p_i$ contributes at least one place above itself that is ramified in $K_1$. Since the total number of ramified primes is exactly five, we conclude that $p_1 \equiv p_2\equiv 5 \pmod 8, q_1 \equiv 7 \pmod 8, q_2 \equiv 3 \pmod 8$.

\smallskip

On the similar lines, we find out that for $K = \Q(\sqrt{q_1q_2q_3q_4})$, ${\rm{rank}} \ A(K) = 2$, ${\rm{rank}} \ A(K^{\prime}) = 3$ and for ${\rm{rank}} \ A(K_1) = 3$ to hold, we need to have either $q_1 \equiv q_2 \equiv q_3 \equiv q_4 \equiv 3 \pmod {8}$ or $q_1 \equiv 7 \pmod {8}$, $q_2 \equiv q_3 \equiv q_4 \equiv 3 \pmod {8}$.

\smallskip

For the field $K = \Q(\sqrt{p_1p_2p_3})$, we follow a similar line of argument to obtain $p_1 \equiv p_2 \equiv p_3 \equiv 5 \pmod 8$ or $p_1 \equiv 1 \pmod 8$, $p_2 \equiv p_3 \equiv 5 \pmod 8$. 

\smallskip

We prove the converse part only for $p_i \equiv 5 \pmod {8}$ and $q_j \equiv 3 \pmod {8}$ as the remaining cases follow similarly. In this case, we observe that the narrow genus field is given by $$K^{+}_G = \Q( \sqrt{p_1p_2q_1q_2},\sqrt{p_1}, \sqrt{p_2}, \sqrt{-q_1}, \sqrt{-q_2}) = \Q(\sqrt{p_1}, \sqrt{p_2}, \sqrt{-q_1}, \sqrt{-q_2}).$$ Hence the genus field $K_G = \Q(\sqrt{p_1}, \sqrt{p_2},\sqrt{q_1q_2})$. Then we have ${\rm{Gal}}(K_G/K) \simeq \Z/2\Z \oplus \Z/2\Z$ and hence ${\rm{rank}} \ A(K)=2$. Likewise, we find that the genus field $K^{\prime}_{G}$ of $K'$ is given by $\Q(\sqrt{2},\sqrt{p_1}, \sqrt{p_2},\sqrt{q_1q_2})$ and ${\rm{rank}}\ A(K') =3$. Again, by using Theorem \ref{rank of biq}, we conclude that ${\rm{rank}} \ A(K_{1}) = 2$. This completes the proof of Theorem \ref{rank stability}. $\hfill\Box$

\section{Proof of Corollary \ref{cor to rank stability}}
For each of the fields mentioned in Theorem \ref{rank stability}, the prime above the rational prime $2$ is unramified in $K$. Hence, the prime above $2\mathcal{O}_K$ is ramified in the extension $K_1/K$. Since $\left[K_1 : K\right] = 2$, the prime above $2\mathcal{O}_K$ is totally ramified. Let $K_n = K\Q_{n}$ be the $n^{\rm{th}}$ layer in the cyclotomic $\Z_2$-extension of $K$. In the tower $\Q_{0} = \Q \subseteq \Q_1 \subseteq \cdots \Q_n \subseteq K_n$, we see that the prime above $2$ is ramified with ramification degree 2 in each extension $\Q_{i}/\Q_{i-1}$ for $i=1,\cdots, n$ and it is unramified in $K_n/\Q_n$. This proves that $2$ is totally ramified in the extension $K_n/K$ for any $n \in \mathbb{N}$. Since $2$ is the only prime that is ramified in the cyclotomic $\Z_2$-extension $K_{\infty}/K$ and it is totally ramified in each extension $K_n/K$, it is totally ramified in $K_{\infty}/K$. 

Using Theorem \ref{rank stability} and Theorem \ref{fukuda's result}, we conclude that ${\rm{rank}} \ A(K) = {\rm{rank}} \  A(K_1) =2$ entails ${\rm{rank}} \ (A_n) = 2$ for all $n \geq 0$. This completes the proof of the Corollary. $\hfill\Box$

\section{Proof of Theorem \ref{A(K1) for K = p1p2q1q2)}}
 
Let us consider $K = \Q(\sqrt{d})= \Q(\sqrt{p_1p_2q_1q_2})$ with $p_1 \equiv p_2 \equiv 5 \pmod 8$, $q_1 \equiv 7 \pmod 8$ and $q_2 \equiv 3 \pmod 8$. Since $d$ has prime factors that are congruent to $3 \pmod 4$, by using Proposition \ref{A(K) and A+(K) are 2 elmtry}, we have $A(K)$ is $2$-elementary if and only if $A^{+}(K)$ is $2$-elementary. Also, we have $S_1(K) = \{ (1, p_1p_2q_1q_2) , \  (p_1, p_2q_1q_2),  \ (p_2, p_1q_1q_2),  \ (-q_1, -p_1p_2q_2), \ (-q_2, -p_1p_2q_1), \ (p_1p_2, q_1q_2),  \\ (-p_1q_1, -p_2q_2),  \ (-p_2q_1, -p_1q_2) \}$.

\smallskip

We note that the order of the appearance of the terms in any of the above pairs can be rearranged depending on which one is bigger in terms of the absolute value. We consider a table (cf. Table $3$) where we list all possible Kronecker symbols corresponding to each pair in order to appeal to Theorem \ref{RR}.  

\begin{table}[hbt!]
\begin{center}
\caption{Kronecker symbols corresponding to each element in $S_1(K)$}
\smallskip
\label{tab:table3}
\begin{tabular}{|c| c| c| } 
\hline 
Sr. No. & Tuple & Kronecker Symbols  \\
\hline
$1$ & $(p_1, p_2q_1q_2)$ & $ \left(\dfrac{p_1}{p_2}\right), \left(\dfrac{p_1}{q_1}\right),\left(\dfrac{p_1}{q_2}\right),\left(\dfrac{p_2q_1q_2}{p_1}\right)$  \\
$2$ & $(p_2, p_1q_1q_2)$  & $\left(\dfrac{p_2}{p_1}\right),\left(\dfrac{p_2}{q_1}\right),\left(\dfrac{p_2}{q_2}\right),\left(\dfrac{p_1q_1q_2}{p_2}\right)$ \\
$3$ & $(-q_1, -p_1p_2q_2)$ & $\left(\dfrac{-q_1}{p_1}\right),\left(\dfrac{-q_1}{p_2}\right),\left(\dfrac{-q_1}{q_2}\right),\left(\dfrac{-p_1p_2q_2}{q_1}\right)$ \\
$4$ & $(-q_2, -p_1p_2q_1)$ & $\left(\dfrac{-q_2}{p_1}\right),\left(\dfrac{-q_2}{p_2}\right),\left(\dfrac{-q_2}{q_1}\right),\left(\dfrac{-p_1p_2q_1}{p_2}\right)$ \\
$5$ & $(p_1p_2, q_1q_2)$ & $\left(\dfrac{p_1p_2}{q_1}\right),\left(\dfrac{p_1p_2}{q_2}\right),\left(\dfrac{q_1q_2}{p_1}\right),\left(\dfrac{q_1q_2}{p_2}\right)$ \\
$6$ & $(-p_1q_1, -p_2q_2)$ & $ \left(\dfrac{-p_1q_1}{p_2}\right),\left(\dfrac{-p_1q_1}{q_2}\right),\left(\dfrac{-p_2q_2}{p_1}\right),\left(\dfrac{-p_2q_2}{q_1}\right)$ \\
$7$ & $(-p_1q_2, -p_2q_1)$ & $ \left(\dfrac{-p_1q_2}{p_2}\right),\left(\dfrac{-p_1q_2}{q_1}\right),\left(\dfrac{-p_2q_1}{p_1}\right),\left(\dfrac{-p_2q_1}{q_2}\right)$ \\
\hline
\end{tabular}
\end{center}
\end{table}

For $\#S_2(K) = 1$, we require at least one entry in each row of Table $3$ to be equal to $-1$. By considering the combinations of the Legendre symbols $\left( \dfrac{p_i}{q_j} \right)$, we find that the any of the following conditions are necessary and sufficient for $\#S_2(K) = 1$
\begin{enumerate}
\item $\left( \dfrac{p_1}{p_2} \right) = -1$ and $\left( \dfrac{q_1q_2}{p_1}\right) = -1$,

\smallskip

\item $\left( \dfrac{p_1}{p_2} \right) = -1$, $\left( \dfrac{q_1q_2}{p_1}\right) = 1$ and $\left( \dfrac{q_1q_2}{p_2}\right) = -1$,

\smallskip

\item $\left( \dfrac{p_1}{p_2} \right) = 1$, $\left( \dfrac{q_1q_2}{p_1}\right) = -1$ and $\left( \dfrac{q_1q_2}{p_2}\right) = -1$.
\end{enumerate}

Therefore, when any one of the above Legendre symbol conditions occurs, we have $\#S_2(K)=1$ and consequently, $A^{+}(K)$ is $2$-elementary. This implies that $\#S_1(K) = \#A^{+}(K) = 8$ and $A(K) \simeq \Z/2Z \oplus \Z/2\Z$. But we simultaneously require $A(K') \simeq \Z/2Z \oplus \Z/2\Z \oplus \Z/2\Z$. Since $D_{K'} = 8p_1p_2q_1q_2$, we see that $\#S_1(K') = 16$. Out of these Legendre symbols, we see that $\#S_2(K') = 1$ if and only if one of the following holds.
\begin{enumerate}
\item $\left( \dfrac{p_1}{p_2} \right) = -1$, $\left( \dfrac{p_1}{q_1} \right) = -1$, $\left( \dfrac{p_1}{q_2} \right) = 1$, $\left( \dfrac{q_1q_2}{p_2} \right) = 1$,
\item $\left( \dfrac{p_1}{p_2} \right) = -1$, $\left( \dfrac{q_1q_2}{p_1} \right) = 1$, $\left( \dfrac{p_2}{q_1} \right) = -1$, $\left( \dfrac{p_2}{q_2} \right) = 1$,
\item $\left( \dfrac{p_1}{p_2} \right) = 1$, $\left( \dfrac{p_1p_2}{q_1} \right) = -1$, $\left( \dfrac{p_1p_2}{q_2} \right) = -1$,  $\left( \dfrac{p_1}{q_1} \right) = \left( \dfrac{p_2}{q_2} \right) $.
\end{enumerate}
We prove the theorem only for the first set of Legendre symbol conditions because all other cases follow a similar line of argument. We first find the decomposition fields for places in $K$ lying above the rational primes $2, p_1, p_2, q_1$ and $q_2$ with respect to the extension $L(K)/K$. We know that a prime splits completely in the number fields $F_1$ and $F_2$, if and only if it splits completely in their compositum $F_1F_2$. In order to find the decomposition field of each place, we first look at the bi-quadratic extensions $F_i$ of $\Q$ such that $\Q \subseteq K \subseteq F_i \subseteq L(K)$. Since $L(K) =K_G = \Q(\sqrt{p_1}, \sqrt{p_2}, \sqrt{q_1q_2})$, we find that $F_1 = \Q(\sqrt{p_1p_2}, \sqrt{q_1q_2}), F_2 = \Q(\sqrt{p_1}, \sqrt{p_2q_1q_2}),$ and $F_3 = \Q(\sqrt{p_2}, \sqrt{p_1q_1q_2})$. Similarly, for $K'$, we find the bi-quadratic subfields $\Q \subseteq K' \subseteq F'_i \subseteq L(K')$ as $F'_1 = \Q( \sqrt{2}, \sqrt{p_1p_2q_1q_2}), F'_2 = \Q( \sqrt{2p_1}, \sqrt{p_2q_1q_2}), F'_3 = \Q( \sqrt{2p_2}, \sqrt{p_1q_1q_2}), F'_4 = \Q( \sqrt{2q_1q_2}, \sqrt{p_1p_2}), F'_5 = \Q( \sqrt{2p_1p_2}, \sqrt{q_1q_2}), F'_6 = \Q( \sqrt{p_1}, \sqrt{2p_2q_1q_2})$, and $F'_7 = \Q( \sqrt{p_2}, \sqrt{2p_1q_1q_2})$. 

\smallskip

Since $p_1p_2q_1q_2 \equiv 5 \pmod 8$, the rational prime $2$ must be inert in $K$, and hence $2\mathcal{O}_K = \mathfrak{l}$ is a prime ideal in $\mathcal{O}_K$. Also by congruence modulo $8$ conditions on the primes, we find that in all subfields of $F_1$ except $K$, $2$ must be totally decomposed. Thus $\mathfrak{l}$ must be totally decomposed in $F_1$. Likewise, $\mathfrak{l}$ must be totally decomposed in $F_2$ and $F_3$. Hence the decomposition field of $\mathfrak{l}$ in $L(K)/K$ is the compositum $F_{1}F_{2}F_{3}$, which is $L(K)$ itself. For the field $K'$, the rational prime $2$ is ramified and we have $2\mathcal{O}_{K'} = \mathfrak{l^{\prime}}^{2}$ where $\mathfrak{l^{\prime}}$ is a prime ideal in $\mathcal{O}_{K'}$. Again from the congruence conditions, we find that  $\mathfrak{l'}$ is totally decomposed only in the fields $F'_2, F'_3$ and $F'_4$. Hence their compositum $F'_2F'_3F'_4$ must be the decomposition field of $\mathfrak{l}'$ in $L(K')/K'$.\\

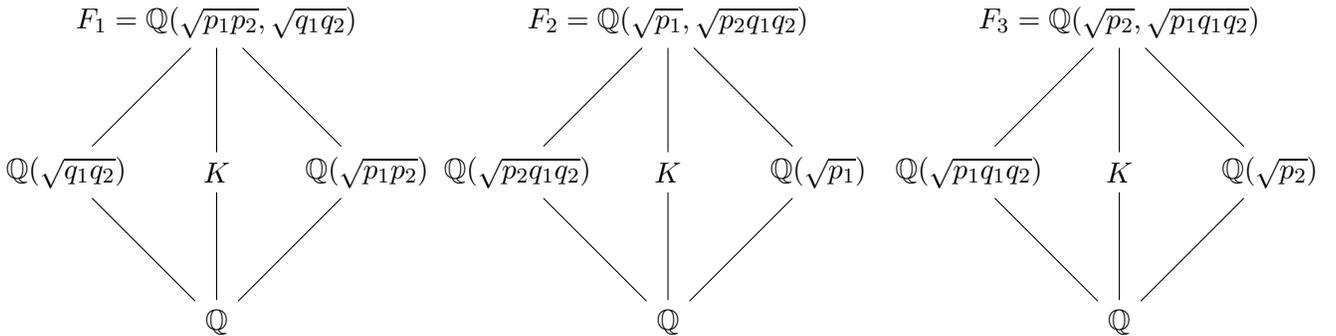
 \begin{figure}[hbt!] 
 \begin{tikzpicture}
 
    \node (Q1) at (0,0) {$\Q$};
    \node (Q2) at (2,2) {$\Q(\sqrt{p_1p_2})$};
    \node (Q3) at (0,2) {$K$};
    \node (Q4) at (-2,2) {$\Q(\sqrt{q_1q_2})$};  
    \node (Q5) at (0,4) {$F_1 = \Q(\sqrt{p_1p_2},\sqrt{q_1q_2})$};

    \draw (Q1)--(Q2);
    \draw (Q1)--(Q3); 
    \draw (Q1)--(Q4);
    \draw (Q2)--(Q5);
    \draw (Q3)--(Q5);
    \draw (Q4)--(Q5);


    \node (P1) at (6,0) {$\Q$};
    \node (P2) at (8,2) {$\Q(\sqrt{p_1})$};
    \node (P3) at (6,2) {$K$};
    \node (P4) at (4,2) {$\Q(\sqrt{p_2q_1q_2})$};  
    \node (P5) at (6,4) {$F_2 = \Q(\sqrt{p_1},\sqrt{p_2q_1q_2})$};

    \draw (P1)--(P2);
    \draw (P1)--(P3); 
    \draw (P1)--(P4);
    \draw (P2)--(P5);
    \draw (P3)--(P5);
    \draw (P4)--(P5);


    \node (R1) at (12,0) {$\Q$};
    \node (R2) at (14,2) {$\Q(\sqrt{p_2})$};
    \node (R3) at (12,2) {$K$};
    \node (R4) at (10,2) {$\Q(\sqrt{p_1q_1q_2})$};  
    \node (R5) at (12,4) {$F_3 =\Q(\sqrt{p_2},\sqrt{p_1q_1q_2})$};

    \draw (R1)--(R2);
    \draw (R1)--(R3); 
    \draw (R1)--(R4);
    \draw (R2)--(R5);
    \draw (R3)--(R5);
    \draw (R4)--(R5);
  
    \end{tikzpicture}
    \end{figure}

 \begin{figure}[hbt!] 
 \begin{tikzpicture}

    \node (Q1) at (0,0) {$\Q$};
    \node (Q2) at (2,2) {$\Q(\sqrt{p_1p_2})$};
    \node (Q3) at (0,2) {$K$};
    \node (Q4) at (-2,2) {$\Q(\sqrt{2q_1q_2})$};  
    \node (Q5) at (0,4) {$F'_4 = \Q(\sqrt{p_1p_2},\sqrt{2q_1q_2})$};
    
    \draw (Q1)--(Q2);
    \draw (Q1)--(Q3); 
    \draw (Q1)--(Q4);
    \draw (Q2)--(Q5);
    \draw (Q3)--(Q5);
    \draw (Q4)--(Q5);


    \node (P1) at (6,0) {$\Q$};
    \node (P2) at (8,2) {$\Q(\sqrt{2p_1})$};
    \node (P3) at (6,2) {$K$};
    \node (P4) at (4,2) {$\Q(\sqrt{p_2q_1q_2})$};  
    \node (P5) at (6,4) {$F'_2 = \Q(\sqrt{2p_1},\sqrt{p_2q_1q_2})$};
    
    \draw (P1)--(P2);
    \draw (P1)--(P3); 
    \draw (P1)--(P4);
    \draw (P2)--(P5);
    \draw (P3)--(P5);
    \draw (P4)--(P5);


    \node (R1) at (12,0) {$\Q$};
    \node (R2) at (14,2) {$\Q(\sqrt{2p_2})$};
    \node (R3) at (12,2) {$K$};
    \node (R4) at (10,2) {$\Q(\sqrt{p_1q_1q_2})$};  
    \node (R5) at (12,4) {$F'_3 =\Q(\sqrt{2p_2},\sqrt{p_1q_1q_2})$};
    
    \draw (R1)--(R2);
    \draw (R1)--(R3); 
    \draw (R1)--(R4);
    \draw (R2)--(R5);
    \draw (R3)--(R5);
    \draw (R4)--(R5);

    \end{tikzpicture}
    \end{figure}

Since $\left( \dfrac{q_1q_2}{p_2} \right) = 1$, we have either $\left( \dfrac{p_2}{q_1} \right) = \left( \dfrac{p_2}{q_2} \right) =1$ or $\left( \dfrac{p_2}{q_1} \right) = \left( \dfrac{p_2}{q_2} \right) = -1$. Let $p_i\mathcal{O}_K = \mathfrak{p}_i^2$ for $i=1,2$ and $q_j\mathcal{O}_K = \mathfrak{q}_j^2$ for $j = 1, 2$. Applying the Legendre symbol conditions and the fact that ramification index and residue degree are multiplicative, we conclude that when $\left( \dfrac{p_2}{q_1} \right) = \left( \dfrac{p_2}{q_2} \right) =1$, the decomposition field of $\mathfrak{q}_2$ is same as that of $\mathfrak{l}$, which equals $L(K)$. Hence the corresponding decomposition groups must be equal (in fact it would be the trivial group as the primes are totally decomposed in $L(K)$). Therefore, the respective Artin symbols of $\mathfrak{q}_2$ and $\mathfrak{l}$ must be equal. That is, $$\left( \dfrac{L(K)/K}{\mathfrak{q}_2} \right) = \left( \dfrac{L(K)/K}{\mathfrak{l}} \right).$$

Hence $\left[\mathfrak{l}\right] = \left[\mathfrak{q}_2\right]  \mbox{ and therefore } \langle \alpha \rangle \mathfrak{q}_2 = \mathfrak{l} = 2\mathcal{O}_K $, for some $\alpha \in K^{\times}$. Squaring both sides, we get $\langle \alpha^2 \rangle q_2\mathcal{O}_K = 4\mathcal{O}_K \mbox{ which implies } 4 = \eps^n \alpha^2 q_2$ for some $n \in \Z$, where $\varepsilon$ is the fundamental unit of $K$. If $n$ is even, then we get $2 = \eps^{\frac{n}{2}}\alpha\sqrt{q}_2$ which yields $\sqrt{q_2} \in K$ and thus $\Q( \sqrt{q_2}) = \Q(\sqrt{p_1p_2q_1q_2})$, which is a contradiction. Therefore, $n$ must be odd. In that case we get $2 = \sqrt{\eps}\beta\sqrt{q_2}$, where $\beta = \eps^{\frac{n-1}{2}}\alpha$. Now, if $\sqrt{\eps} \in K_1$, then $K_1= K_1(\sqrt{\eps}) = K_1(\sqrt{q_2})$, which is again not possible as $K_1 \neq K_1(\sqrt{q_2})$. Therefore, $\sqrt{\eps}\not\in K_1$ and $K_1(\sqrt{\eps}) = K_1(\sqrt{q_2})$. 

\smallskip

If we have $\left( \dfrac{p_2}{q_1} \right) = \left( \dfrac{p_2}{q_2} \right) =-1$, then we get that the decomposition field of $\mathfrak{q}_1$ is same as that of $\mathfrak{p}_2$. Correspondingly, we get that $\langle {\alpha_1}^2 \rangle p_2\mathcal{O}_K = q_1\mathcal{O}_K \Rightarrow p_2 = \eps^n {\alpha_1}^2 q_1$ for some $n \in \Z$ and $\alpha_1 \in K^{\times}$. If $n$ is even then $\sqrt{p_2} = \eps^{\frac{n}{2}}\alpha_1\sqrt{q_1}$, and thus $K(\sqrt{p_2}) = K(\sqrt{q_1})$, which is not possible. Therefore, $n$ must be odd. In that case, $\sqrt{p_2} = \sqrt{\eps}\beta_1\sqrt{q_1}$ where $\beta_1 \in K^{\times}$. If $\sqrt{\eps} \in K_1$, then we obtain that $K_1(\sqrt{p_2}) = K_1(\sqrt{q_1})$ which is not true. Therefore, $\sqrt{\eps} \not\in K_1$ and also, $K_1(\sqrt{\eps}) = K_1\left(\sqrt{\frac{p_2}{q_1}}\right) = K_1(\sqrt{p_1q_2})$.

\smallskip

For the field $K'$, let $p_i\mathcal{O}_{K'} = {{\mathfrak{p}'}_i}^2$ for $i=1,2$ and $q_j\mathcal{O}_{K'} = {{\mathfrak{q}}'_j}^2$ for $j = 1, 2$.  Irrespective of whether $\left( \dfrac{p_2}{q_1} \right) = \left( \dfrac{p_2}{q_2} \right) =\pm 1$, we find the decomposition field of ${\mathfrak{p}'}_1$ and $\mathfrak{l}'$ to be equal, arguing as before. Proceeding as above, we find that $\eps' \not\in K_1$ and $K_1(\sqrt{\eps'}) = K_1(\sqrt{p_1})$. Also, we note that $K_1(\sqrt{\eps}) \neq K_1(\sqrt{\eps'})$. If $\sqrt{\eps \eps'} \in K_1$, then
$\sqrt{\eps \eps'} \in K_1 \subseteq K_1(\sqrt{\eps})$. This implies $\sqrt{\eps'} \in K_1(\sqrt{\eps})$ and consequently, we have $K_1(\sqrt{\eps'}) \subseteq K_1(\sqrt{\eps})$. Now, $$\sqrt{\eps \eps'} \in K_1 \subseteq K_1(\sqrt{\eps'}) \Rightarrow \sqrt{\eps} \in K_1(\sqrt{\eps'}) \Rightarrow K_1(\sqrt{\eps}) \subseteq K_1(\sqrt{\eps'}) \Rightarrow K_1(\sqrt{\eps}) = K_1(\sqrt{\eps'}),$$ which is a contradiction.

\smallskip

Therefore, we conclude that $\sqrt{\eps}, \sqrt{\eps'}, \sqrt{\eps\eps'} \not\in K_1$, which means that any system of fundamental units of $K_1$ does not contain the square-roots of the fundamental units of $K$ or $K'$, nor does it contain the product of their square roots. We note that the subfields of $K_1$ are $K, K'$ and $\Q(\sqrt{2})$. The fundamental unit of $\Q(\sqrt{2})$ is $1 + \sqrt{2}$ and it has norm $N_{\Q(\sqrt{2})/\Q}(1 + \sqrt{2}) =-1$. From Theorem \ref{kubota}, we deduce that a system of fundamental units of $K_1$ must be $\{ \eps, \eps', 1+ \sqrt{2} \}$. Hence the Hasse unit index $Q(K_1) = 1$. Since the class number of $\Q(\sqrt{2}) \mbox{ equals } 1$, we have $\# A(\Q(\sqrt{2})) = 1 $, and by Theorem \ref{kubota}, we obtain
 $$\#A(K_1) = \dfrac{1}{4}\cdot Q(K_1)\cdot \#A(K)\cdot \#A(K')\cdot \#A(\Q(\sqrt{2})) = \dfrac{1}{4}\cdot 1 \cdot 4 \cdot 8 \cdot 1 = 8.$$ 
Since ${\rm{rank}} \ A(K_1) = 2$, and $\#A(K_1) = 8$, we have $A(K_1) \simeq \Z/2\Z \oplus \Z/4\Z$. This completes the proof of case 1 of Theorem \ref{A(K1) for K = p1p2q1q2)}. $\hfill\Box$

\smallskip

{\bf Acknowledgements.} The authors take immense pleasure to thank Indian Institute of Technology Guwahati for providing excellent facilities to carry out this research. The first author sincerely acknowledges the National Board of Higher Mathematics (NBHM) for the Post-Doctoral Fellowship (Order No. 0204/16(12)/2020/R \& D-II/10925). The research of the third author is partially funded by the MATRICS, SERB research grant MTR/2020/000467.

\end{document}